\documentclass[12pt,reqno,oneside]{amsart}
\usepackage{amsfonts}
\usepackage{amsmath, amssymb}
\usepackage{hyperref}

\numberwithin{equation}{section} 
\newtheorem{theorem}{Theorem}[section]

\newtheorem{lemma}[theorem]{Lemma}
\theoremstyle{definition}
\newtheorem{definition}[theorem]{Definition}
\theoremstyle{remark}

\numberwithin{equation}{section}

\input{tcilatex}

\begin{document}
\thanks{$^{\ast }$corresponding author.}
\title[Toeplitz Matrices-with $q$-derivative operator]{Construction of Toeplitz Matrices whose elements are the coefficients
of univalent functions associated with $q$-derivative operator}
\author{Nanjundan Magesh$^{1}$, \c{S}ahsene Alt\i nkaya$^{2,\ast }$, Sibel
Yal\c{c}\i n$^{2}$}
\address{$^{1}$Post-Graduate and Research Department of Mathematics\\
Government Arts College for Men\\
Krishnagiri 635001, Tamilnadu, India.}
\email{nmagi\_2000@yahoo.co.in}
\address{$^{2}$Department of Mathematics, Faculty of Arts and Science,\\ Uludag University, 16059, Bursa, Turkey.}
\email{sahsene@uludag.edu.tr, syalcin@uludag.edu.tr}
\maketitle

\begin{abstract}
In this paper, we find the coefficient bounds using symmetric Toeplitz determinants for the functions belonging to the subclass $R(q)$.
\
\end{abstract}

\section{Introduction, Definitions and Notations}

Let $A$ indicate an analytic function family, which is normalized under the
condition of $f(0)=f^{\prime }(0)-1=0$ in $\Delta =\left\{ z:z\in
\mathbb{C}
\text{ and }\left\vert z\right\vert <1\right\} $ and given by the following
Taylor-Maclaurin series:
\begin{equation}
f(z)=z+\overset{\infty }{\underset{n=2}{\sum }}a_{n}z^{n}.  \label{eq1}
\end{equation}%
Also let $S$ be the subclass of $A$ consisting functions of the form (\ref{eq1}) which
are also univalent in $\Delta .$

Let $f$ be given by (\ref{eq1}). Then $f\in RT$ if it satisfies the
inequality%
\begin{equation*}
\Re \left( f^{\prime }(z)\right) >0,\ \ \ \ \ z\in \Delta .
\end{equation*}%
The subclass $R$ was studied systematically by MacGregor \cite{mac} who
indeed referred to numerous earlier investigations involving functions whose
derivative has a positive real part.

In the field of Geometric Functions Theory, various subclasses of analytic
functions have been studied from different viewpoints. The fractional
\textit{q}-calculus is the important tools that are used to investigate
subclasses of analytic functions. For example, the extension of the theory
of univalent functions can be described by using the theory of \textit{q}%
-calculus. Moreover, the \textit{q}-calculus operators, such as fractional
\textit{q}-integral and fractional \textit{q}-derivative operators, are used
to construct several subclasses of analytic functions (see, e.g., \cite%
{Aydogan 2013,Polatoglu 2016,Ucar 2016}). In a recent paper
Purohit and Raina \cite{Purohit 2013}, investigated applications of
fractional \textit{q}-calculus operators to defined certain new classes of
functions which are analytic in the open disk. Later, Mohammed and Darus
\cite{Mohammed 2013} studied approximation and geometric properties of these
\textit{q}-operators in some subclasses of analytic functions in compact
disk.

For the convenience, we provide some basic definitions and concept details
of \textit{q}-calculus which are used in this paper. We suppose throughout
the paper that $0<q<1.$ We shall follow the notation and terminology as in
\cite{Gasper 90}. We recall the definitions of fractional \textit{q}-calculus
operators of complex valued function $f(z).$

\begin{definition}
Let $q\in (0,1)$ and define%
\begin{equation*}
\left[ n\right] _{q}=\frac{1-q^{n}}{1-q},
\end{equation*}%
for $n\in
\mathbb{N}
.$\
\end{definition}

\begin{definition}
(see \cite{Jackson 08}) The q-derivative of a function $f$, defined on a
subset of $\mathbb{C}$, is given by
\begin{equation*}
(D_{q}f)(z)=\left\{
\begin{array}{lll}
\dfrac{f(z)-f(qz)}{(1-q)z} & for & z\neq 0, \\
&  &  \\
f^{\prime }(0) & for & z=0.%
\end{array}%
\right.
\end{equation*}
\end{definition}

We note that $\lim\limits_{q\rightarrow 1^{-}}(D_{q}f)(z)=f^{\prime }(z)$ if
$f$\ is differentiable at $z$. Additionally, in view of (\ref{eq1}), we
deduce that
\begin{equation}
(D_{q}f)(z)=1+\overset{\infty }{\underset{n=2}{\sum }}\left[ n\right]
_{q}a_{n}z^{n-1}.  \label{eq3}
\end{equation}

Toeplitz matrices frequently arise in many application areas and have been
attracted much attention in recent years. They arise in pure mathematics:
algebra, algebraic geometry, analysis, combinatorics, differential geometry,
as well as in applied mathematics: approximation theory, compressive
sensing, numerical integral equations, numerical integration, statistics,
time series analysis, and among other areas (see, for example \cite{ye}).

The symmetric Toeplitz determinant of $f$ \ for $n\geq 1$ and $q\geq 1$ is
defined by Thomas and Halim \cite{t}, as follows:%
\begin{equation*}
T_{q}(n)=\left\vert
\begin{array}{llll}
a_{n} & a_{n+1} & \cdots  & a_{n+q-1} \\
a_{n+1} & a_{n} & \cdots  & a_{n+q-2} \\
\vdots  & \vdots  & \vdots  & \vdots  \\
a_{n+q-1} & a_{n+q-2} & \cdots  & a_{n}%
\end{array}%
\right\vert \ \ \ \ \ \ (a_{1}=1).
\end{equation*}

Note that%
\begin{equation*}
T_{2}(2)=\left\vert
\begin{array}{ll}
a_{2} & a_{2} \\
a_{2} & a_{2}%
\end{array}%
\right\vert ,\ \ T_{2}(3)=\left\vert
\begin{array}{ll}
a_{3} & a_{4} \\
a_{4} & a_{3}%
\end{array}%
\right\vert ,
\end{equation*}%
and%
\begin{equation*}
T_{3}(1)=\left\vert
\begin{array}{ccc}
1 & a_{2} & a_{3} \\
a_{2} & 1 & a_{2} \\
a_{3} & a_{2} & 1%
\end{array}%
\right\vert ,\ \ T_{3}(2)=\left\vert
\begin{array}{ccc}
a_{2} & a_{3} & a_{4} \\
a_{3} & a_{2} & a_{3} \\
a_{4} & a_{3} & a_{2}%
\end{array}%
\right\vert .
\end{equation*}%
Very recently, the estimates of the Toeplitz determinant $\left\vert
T_{q}(n)\right\vert $ for functions in $RT$ have been studied in \cite{ra}.

\section{Preliminaries}

Let $P$ be the class of functions with positive real part consisting of all
analytic functions $p:\Delta \rightarrow
\mathbb{C}
$ satisfying $p(0)=1$ and $\Re (p(z))>0.$ The class $P$ is called the class
of Caratheodory function.

The following results will be required for proving our results.

\begin{lemma}\label{le-Pom75}
\cite{Pommerenke 75} If the function $p\in P$, then%
\begin{equation*}
\left\vert p_{n}\right\vert \leq 2~\ \ \ \ \ \left( n\in
\mathbb{N}
=\left\{ 1,2,\ldots \right\} \right)
\end{equation*}%
and%
\begin{equation*}
\left\vert p_{2}-\frac{p_{1}^{2}}{2}\right\vert \leq 2-\frac{\left\vert
p_{1}\right\vert ^{2}}{2}.\
\end{equation*}
\end{lemma}

\begin{lemma}\label{le-gr}
\cite{gr} If the function $p\in P$, then%
\begin{eqnarray*}
2p_{2} &=&p_{1}^{2}+x(4-p_{1}^{2}) \\
&& \\
4p_{3}
&=&p_{1}^{3}+2(4-p_{1}^{2})p_{1}x-p_{1}(4-p_{1}^{2})x^{2}+2(4-p_{1}^{2})(1-%
\left\vert x\right\vert ^{2})z
\end{eqnarray*}%
for some x, z with $\left\vert x\right\vert \leq 1$ and $\left\vert
z\right\vert \leq 1.$
\end{lemma}

\begin{definition}
A function $f\in A$ is said to be in the class $R(q),$ if the following
condition holds%
\begin{equation*}
\Re \left( D_{q}f\right) (z)>0,\ \ \ \ \ z\in \Delta .
\end{equation*}
\end{definition}

We note that%
\begin{equation*}
\lim_{q\rightarrow 1^{-}}R(q)=\left\{ f\in A:\lim_{q\rightarrow 1^{-}}\Re
\left( D_{q}f\right) (z)>0,\ \ z\in \Delta \right\} =RT.
\end{equation*}

The aim of this work is to obtain the coefficient bounds using symmetric
Toeplitz determinants $T_{2}(2),$ $T_{2}(3),$ $T_{3}(2)$ and $T_{3}(1)$ for
the functions belonging to the subclass $R(q)$.

\section{Main Results and Their Consequences}

\begin{theorem}
Let $\ f$ given by (\ref{eq1}) be in the class $R(q).$Then%
\begin{equation*}
\left\vert T_{2}(2)\right\vert \leq \frac{4q^{2}(q^{2}+2q+2)}{%
(1+2q+2q^{2}+q^{3})^{2}}.
\end{equation*}
\end{theorem}

\begin{proof}
Let $\ f\in R(q).$ Then there exists a function $p\in P$ such that%
\begin{equation}
\left( D_{q}f\right) (z)=p(z).  \label{e3}
\end{equation}%
By equating the coefficients, we obtain
\begin{equation}
a_{2}=\frac{p_{1}}{\left[ 2\right] _{q}},  \label{eq9}
\end{equation}%
\begin{equation}
a_{3}=\frac{p_{2}}{\left[ 3\right] _{q}},  \label{eq10}
\end{equation}%
and%
\begin{equation}
a_{4}=\frac{p_{3}}{\left[ 4\right] _{q}}.  \label{eq11}
\end{equation}%
It follows from (\ref{eq9}), (\ref{eq10}) and (\ref{eq11}) that%
\begin{equation*}
\left\vert T_{2}(2)\right\vert =\left\vert a_{3}^{2}-a_{2}^{2}\right\vert
=\left\vert \frac{p_{2}^{2}}{\left[ 3\right] _{q}^{2}}-\frac{p_{1}^{2}}{%
\left[ 2\right] _{q}^{2}}\right\vert .
\end{equation*}%
Making use of Lemma \ref{le-gr} to express $p_{2}$ in terms of $p_{1}$, we obtain%
\begin{equation*}
\left\vert a_{3}^{2}-a_{2}^{2}\right\vert =\left\vert \frac{p_{1}^{4}}{4%
\left[ 3\right] _{q}^{2}}-\frac{p_{1}^{2}}{\left[ 2\right] _{q}^{2}}+\frac{%
xp_{1}^{2}(4-p_{1}^{2})}{2\left[ 3\right] _{q}^{2}}+\frac{%
x^{2}(4-p_{1}^{2})^{2}}{4\left[ 3\right] _{q}^{2}}\right\vert .
\end{equation*}%
By Lemma \ref{le-Pom75}, we have $\left\vert p_{1}\right\vert \leq 2$. For convenience of
notation, we take $p_{1}=p$ and we may assume without loss of generality
that $p\in \lbrack 0,2]$. Applying the triangle inequality with $P=(4-p^{2})$%
, we get
\begin{equation*}
\left\vert a_{3}^{2}-a_{2}^{2}\right\vert \leq \left\vert \frac{p^{4}}{4%
\left[ 3\right] _{q}^{2}}-\frac{p^{2}}{\left[ 2\right] _{q}^{2}}\right\vert +%
\frac{\left\vert x\right\vert p^{2}P}{2\left[ 3\right] _{q}^{2}}+\frac{%
\left\vert x\right\vert ^{2}P^{2}}{4\left[ 3\right] _{q}^{2}}=:F(\left\vert
x\right\vert ).
\end{equation*}%
Differentiating $F(\left\vert x\right\vert )$, one can see clearly that $%
F^{\prime }(\left\vert x\right\vert )>0$ on $[0,1],$ and so $F(\left\vert
x\right\vert )\leq F(1)$. Hence
\begin{equation*}
F(\left\vert x\right\vert )\leq F(1)=\left\vert \frac{p^{4}}{4\left[ 3\right]
_{q}^{2}}-\frac{p^{2}}{\left[ 2\right] _{q}^{2}}\right\vert +\frac{1}{\left[
3\right] _{q}^{2}}\left( 4-\frac{p^{4}}{4}\right) .
\end{equation*}%
Treating the cases when the absolute term is either positive or negative, we
can show that this expression $F(\left\vert x\right\vert )$ has a maximum
value $\frac{p^{2}}{\left[ 2\right] _{q}^{2}}-\frac{p^{4}}{4\left[ 3\right]
_{q}^{2}}$ on $[0,2],$ when $p=2.$
\end{proof}

\begin{theorem}
Let $\ f$ given by (\ref{eq1}) be in the class $R(q).$Then%
\begin{equation*}
\left\vert T_{2}(3)\right\vert \leq \frac{16q^{2}}{%
(1+q+q^{2}+q^{3})(1+q+q^{2})^{2}}.
\end{equation*}
\end{theorem}

\begin{proof}
Using (\ref{eq9}), (\ref{eq10}), (\ref{eq11}) and Lemma \ref{le-gr} to express $p_{2}$
and $p_{3}$ in terms of $p_{1}$, we obtain, with $P=(4-p^{2})$ and $%
R=(1-\left\vert x\right\vert ^{2})z$,%
\begin{eqnarray*}
\left\vert a_{4}^{2}-a_{3}^{2}\right\vert  &=&\frac{1}{4}\left\{ \left(
\frac{p^{2}P^{2}}{4\left[ 4\right] _{q}^{2}}+\frac{P^{2}}{\left[ 4\right]
_{q}^{2}}-\frac{pP^{2}}{\left[ 4\right] _{q}^{2}}\right) \left\vert
x\right\vert ^{4}+\left( \frac{p^{2}P^{2}}{\left[ 4\right] _{q}^{2}}-\frac{%
2pP^{2}}{\left[ 4\right] _{q}^{2}}\right) \left\vert x\right\vert
^{3}\right.  \\
&& \\
&&\left. +\left( \frac{p^{2}P^{2}}{\left[ 4\right] _{q}^{2}}-\frac{2P^{2}}{%
\left[ 4\right] _{q}^{2}}+\frac{P^{2}}{\left[ 3\right] _{q}^{2}}+\frac{p^{4}P%
}{2\left[ 4\right] _{q}^{2}}-\frac{p^{3}P}{\left[ 4\right] _{q}^{2}}+\frac{%
pP^{2}}{\left[ 4\right] _{q}^{2}}\right) \left\vert x\right\vert ^{2}\right.
\\
&& \\
&&\left. +\left( \frac{p^{4}P}{\left[ 4\right] _{q}^{2}}+\frac{2pP^{2}}{%
\left[ 4\right] _{q}^{2}}+\frac{2p^{2}P}{\left[ 3\right] _{q}^{2}}\right)
\left\vert x\right\vert \right.  \\
&& \\
&&\left. +\frac{P^{2}}{\left[ 4\right] _{q}^{2}}+\frac{p^{3}P}{\left[ 4%
\right] _{q}^{2}}+\left\vert \frac{p^{6}}{4\left[ 4\right] _{q}^{2}}-\frac{%
p^{4}}{\left[ 3\right] _{q}^{2}}\right\vert \right\}  \\
&& \\
&&:=G(p,\left\vert x\right\vert ).
\end{eqnarray*}%
Differentiating and using a simple calculus shows that $\frac{\partial
G(p,\left\vert x\right\vert )}{\partial \left\vert x\right\vert }\geq 0$ for
$\left\vert x\right\vert \in \left[ 0,1\right] $ and fixed $p\in \left[ 0,2%
\right] .\ $It follows that $G(p,\left\vert x\right\vert )$ is an increasing
function of $\left\vert x\right\vert $. So $G(p,\left\vert x\right\vert
)\leq G(p,1)$. Upon letting $\left\vert x\right\vert =1$, a simple algebraic
manipulation yields%
\begin{equation*}
\left\vert a_{4}^{2}-a_{3}^{2}\right\vert \leq \frac{4}{(1+q+q^{2})^{2}}.
\end{equation*}
\end{proof}

\begin{theorem}
Let $\ f$ given by (\ref{eq1}) be in the class $R(q).$Then%
\begin{equation*}
\left\vert T_{3}(2)\right\vert \leq \frac{16q^{2}}{%
(1+q+q^{2}+q^{3})(1+q+q^{2})^{2}}.
\end{equation*}
\end{theorem}

\begin{proof}
Using the same techniques as in Theorem 6 and Theorem 7, we obtain%
\begin{equation*}
\left\vert a_{2}-a_{4}\right\vert \leq \left\vert \frac{p}{\left[ 2\right]
_{q}}-\frac{p^{3}}{4\left[ 4\right] _{q}}\right\vert +\frac{\left\vert
x\right\vert pP}{2\left[ 4\right] _{q}}+\frac{\left\vert x\right\vert
^{2}(p-2)P}{4\left[ 4\right] _{q}}+\frac{P}{2\left[ 4\right] _{q}}.
\end{equation*}%
It is easy exercise to show that $\left\vert a_{2}-a_{4}\right\vert \leq
\frac{p}{\left[ 2\right] _{q}}-\frac{p^{3}}{4\left[ 4\right] _{q}}$ on $%
[0,2],$ when $p=2.$ Thus,%
\begin{equation}
\left\vert a_{2}-a_{4}\right\vert \leq \frac{2q^{2}}{1+q+q^{2}+q^{3}}.
\label{A}
\end{equation}%
Using the same techniques as above, one can obtain with simple computations
that
\begin{equation}
\left\vert a_{2}^{2}-2a_{3}^{2}+a_{2}a_{4}\right\vert \leq \frac{8}{%
(1+q+q^{2})^{2}}.  \label{B}
\end{equation}%
Finally, from (\ref{A}) and (\ref{B}) we obtain%
\begin{equation*}
\left\vert T_{3}(2)\right\vert =\left\vert
(a_{2}-a_{4})(a_{2}^{2}-2a_{3}^{2}+a_{2}a_{4})\right\vert \leq \frac{16q^{2}%
}{(1+q+q^{2}+q^{3})(1+q+q^{2})^{2}}.
\end{equation*}
\end{proof}

\begin{theorem}
Let $\ f$ given by (\ref{eq1}) be in the class $R(q).$Then%
\begin{equation*}
\left\vert T_{3}(1)\right\vert \leq 1+\frac{4}{(1+q+q^{2})^{2}}.
\end{equation*}
\end{theorem}

\begin{proof}
Expanding the determinant by using (\ref{eq9}), (\ref{eq10}), (\ref{eq11})
and Lemma \ref{le-gr}, we obtain%
\begin{eqnarray*}
\left\vert 1+2a_{2}^{2}(a_{3}-1)-a_{3}^{2}\right\vert  &=&\left\vert 1+2%
\frac{p_{1}^{2}}{\left[ 2\right] _{q}^{2}}\left( \frac{p_{2}}{\left[ 3\right]
_{q}}-1\right) -\frac{p_{2}^{2}}{\left[ 3\right] _{q}^{2}}\right\vert  \\
&& \\
&=&\left\vert 1+\left( \tfrac{1}{\left[ 2\right] _{q}^{2}\left[ 3\right] _{q}%
}-\tfrac{1}{4\left[ 3\right] _{q}^{2}}\right) p_{1}^{4}-2\tfrac{p_{1}^{2}}{%
\left[ 2\right] _{q}^{2}}+\left( \tfrac{1}{\left[ 2\right] _{q}^{2}\left[ 3%
\right] _{q}}-\tfrac{1}{4\left[ 3\right] _{q}^{2}}\right) p_{1}^{2}xP-\tfrac{%
x^{2}P^{2}}{4\left[ 3\right] _{q}^{2}}\right\vert .
\end{eqnarray*}%
As before, without loss in generality we can assume that $p_{1}=p$, where $%
p\in \lbrack 0,2]$. Then, by using the triangle inequality and the fact that
$\left\vert x\right\vert \leq 1$ we obtain%
\begin{equation*}
\left\vert T_{3}(1)\right\vert \leq \left\vert 1+\left( \tfrac{1}{\left[ 2%
\right] _{q}^{2}\left[ 3\right] _{q}}-\tfrac{1}{4\left[ 3\right] _{q}^{2}}%
\right) p^{4}-2\tfrac{p^{2}}{\left[ 2\right] _{q}^{2}}\right\vert +\left(
\tfrac{1}{\left[ 2\right] _{q}^{2}\left[ 3\right] _{q}}-\tfrac{1}{4\left[ 3%
\right] _{q}^{2}}\right) p^{2}(4-p^{2})+\tfrac{(4-p^{2})^{2}}{4\left[ 3%
\right] _{q}^{2}}.
\end{equation*}%
It is now a simple exercise in elementary calculus to show that this
expression has a maximum value when $p=2$, which completes the proof.
\end{proof}

\end{document}